\documentclass[11pt,letterpaper]{amsart}

\usepackage[T1]{fontenc}
\usepackage[utf8]{inputenc}
\usepackage{amssymb,dsfont}
\usepackage[mathscr]{eucal}
\usepackage{enumitem}
\usepackage{graphicx,xcolor}
\usepackage{xparse}
\usepackage[foot]{amsaddr}
\usepackage[colorlinks=true,bookmarksdepth=2]{hyperref}

\numberwithin{equation}{section}

\newtheorem{theorem}{Theorem}[section]
\newtheorem{lemma}[theorem]{Lemma}
\newtheorem{corollary}[theorem]{Corollary}
\newtheorem{proposition}[theorem]{Proposition}

\theoremstyle{definition}
\newtheorem{definition}[theorem]{Definition}

\newtheorem{remark}[theorem]{Remark}

\DeclareMathOperator{\sign}{sign}

\newcommand{\dht}{\mathscr{H}}
\newcommand{\rtt}{\mathscr{R}}
\newcommand{\kht}{\mathscr{K}}
\newcommand{\iii}{\mathscr{I}}
\newcommand{\ttt}{\mathscr{T}}

\newcommand{\fourier}{\mathscr{F}}

\newcommand{\R}{\mathds{R}}
\newcommand{\Z}{\mathds{Z}}
\newcommand{\N}{\mathds{N}}

\newcommand{\skel}{\mathds{S}}
\newcommand{\fram}{\mathds{F}}

\newcommand{\ph}{\varphi}

\renewcommand{\le}{\leqslant}
\renewcommand{\ge}{\geqslant}

\NewDocumentCommand{\formula}{ssom}{%
 \IfBooleanTF{#1}{%
  \IfBooleanTF{#2}{%
   \IfValueTF{#3}%
    {\begin{align}\label{#3}\begin{gathered}#4\end{gathered}\end{align}}%
    {\begin{gather}#4\end{gather}}%
  }{%
   \IfValueTF{#3}%
    {\begin{align}\label{#3}\begin{aligned}#4\end{aligned}\end{align}}%
    {\begin{gather*}#4\end{gather*}}%
  }%
 }{%
  \IfValueTF{#3}%
   {\begin{align}\label{#3}#4\end{align}}%
   {\begin{align*}#4\end{align*}}%
 }%
}

\begin{document}

\title[The $\ell^p$ norm of the Riesz--Titchmarsh transform]{The $\ell^p$ norm of the Riesz--Titchmarsh transform for even integer $p$}

\author{Rodrigo Ba\~nuelos}
\address{Rodrigo Ba\~nuelos \\ Department of Mathematics \\ Purdue University \\ 150 N.\@ University Street, West Lafayette, IN 47907}
\email{banuelos@purdue.edu}

\author{Mateusz Kwa\'snicki}
\address{Mateusz Kwa\'snicki \\ Faculty of Pure and Applied Mathematics \\ Wroc\l aw University of Science and Technology \\ ul.\@ Wybrze\.ze Wyspia\'nskiego 27 \\ 50-370 Wroc\l aw, Poland}
\email{mateusz.kwasnicki@pwr.edu.pl}

\thanks{R.~Ba\~nuelos is supported in part by NSF Grant 1403417-DMS. M.~Kwa\'snicki is supported by the Polish National Science Centre (NCN) grant 2019/33/B/ST1/03098.}

\date{\today}

\keywords{Discrete Hilbert transform, Riesz--Titchmarsh operator, sharp inequalities}
\subjclass[2010]{Primary: 42A50, 42A05. Secondary: 39A12.}

\begin{abstract}
The long-standing conjecture that for $p \in (1, \infty)$ the $\ell^p(\Z)$ norm of the Riesz--Titchmarsh discrete Hilbert transform is the same as the $L^p(\R)$ norm of the classical Hilbert transform, is verified when $p = 2 n$ or $\frac{p}{p - 1} = 2 n$, for $n \in \N$. The proof, which is algebraic in nature, depends in a crucial way on the sharp estimate for the $\ell^p(\Z)$ norm of a different variant of this operator for the full range of $p$. The latter result was recently proved by the authors in~\cite{BanKwa}.
\end{abstract}

\maketitle

%
%

\section{Introduction and statement of main result}

There are several non-equivalent definitions of the \emph{discrete Hilbert transform} on $\Z$. In this paper we focus on the one introduced by E.C.~Titchmarsh in~\cite{Tit26} and often called the Riesz--Titchmarsh operator:
\formula[]{
 \rtt a_n & = \frac{1}{\pi} \sum_{m \in \Z} \frac{a_{n - m}}{m + \tfrac12} \, .
}
A conjecture that has remained open for almost a century states that for $p \in (1, \infty)$ the norm of $\rtt$ on $\ell^p(\Z)$ equals the norm of the (continuous) Hilbert transform
\formula[]{
 H f(x) & = \operatorname{p{.}v{.}} \frac{1}{\pi} \int_{-\infty}^{\infty} \frac{f(x - y)}{y} \, dy
}
on $L^p(\R)$. In~\cite{Tit26} Titchmarsh claimed to have proved this equality, but in the correction note~\cite{Tit27} he points out that his argument only gives the inequality
\formula[eq:titchmarsh]{
 \lVert H \rVert_p & \le \lVert \rtt \rVert_p .
}
The norm of $H$ was found by S.~Pichorides in~\cite{Pic} and by B.~Cole (see~\cite{Gam}) to be $\cot \frac{\pi}{2 p}$ when $p \ge 2$ and $\tan \frac{\pi}{2 p}$ when $p \le 2$. That is,
\formula[eq:pichorides]{
 \lVert H \rVert_p & = \cot \frac{\pi}{2p^*},
}
where $p^* = \max\{p, \tfrac{p}{p - 1}\}$ and $\tfrac{p}{p - 1}$ is the conjugate exponent of $p$. Combining~\eqref{eq:titchmarsh} and~\eqref{eq:pichorides}, the conjecture is to show that
\formula[]{
 \lVert \rtt \rVert_p & \le \cot \frac{\pi}{2 p^*} \, .
}
I.~Ts.~Gohberg, N.~Ya.~Krupnik and V.~I.~Matsaev proved the above inequality when $p$ is a power of $2$ or, by duality, when $p$ is the conjugate exponent of the power of $2$; see equation~(9) in~\cite{GohKru} and equation~(III.10.14) in~\cite{GohKre}. Our main result extends this to the case when $p$ is an even integer, or its conjugate exponent. More precisely, we prove the following result.

\begin{theorem}
\label{thm:main}
For $p = 2 n$ or $p = \frac{2 n}{2 n - 1}$, $n \in \N$, we have
\formula{
 \lVert (\rtt a_n) \rVert_p & \le \cot \frac\pi{2p^*} \, \lVert (a_n) \rVert_p,
}
for every sequence $(a_n)$ in $\ell^p$. The constant is best possible. In particular, the operator norm of $\rtt$ on $\ell^p$ is equal to the operator norm of $H$ on $L^p(\R)$ for such $p$.
\end{theorem}

Although our proof is algebraic, it relies on the sharp $\ell^p$ bound for the original discrete Hilbert transform as introduced by Hilbert in~1909, defined by
\formula[]{
 \dht_0 a_n & = \frac{1}{\pi} \sum_{m \in \Z \setminus \{0\}} \frac{a_{n - m}}{m} \, .
}
This bound was proved recently by the authors in~\cite{BanKwa} for every $p \in (1, \infty)$ using analytical and probabilistic techniques. For the reader's convenience, we state this result explicitly.

\begin{theorem}[\cite{BanKwa}, Theorem~1.1]
\label{thm:bk}
If $p \in (1, \infty)$, the $\ell^p$ norm of the operator $\dht_0$ is given by
\formula[eq:bk]{
 \lVert \dht_0 \rVert_p & = \cot \frac\pi{2p^*} .
}
In particular, the operator norm of $\dht_0$ on $\ell^p$ is equal to the operator norm of $H$ on $L^p(\R)$.
\end{theorem}

We stress that in the proof of Theorem~\ref{thm:main} we need the assertion of Theorem~\ref{thm:bk} for all $p \in (1, \infty)$ and not just for even integers $p$. Formula~\eqref{eq:bk} when $p$ is a power of~2 has been known for several years, see~\cite{Lae}. We refer to Remarks~\ref{rem:k2} and~\ref{rem:k3} below for further discussion.

We note, more generally, that one can consider the one-parameter family of operators defined by
\formula[]{
 T_t a_n & =
 \begin{cases}
  \displaystyle \frac{\sin(\pi t)}{\pi} \sum_{m \in \Z} \frac{a_{n - m}}{m + t} & \text{if $t \in \R \setminus \Z$,} \\
  (-1)^t a_{n + t} & \text{if $t \in \Z$.}
 \end{cases}
}
It is shown in Theorem~1.1 in~\cite{CarSam} that for $t \ge 0$ these operators form a strongly continuous group of isometries on $\ell^2$ with generator $\pi \dht_0$. Note that when $t = \tfrac12$, $T_t$ corresponds to the Riesz--Titchmarsh operator. In Theorem~5.6 in~\cite{Lae}, the $\ell^p$ norm of these operators is studied for $t \in (0, 1)$ and $p \in (1, \infty)$, and it is proved that
\formula[eq:laeng]{
 \lVert T_t \rVert_p & \ge \lVert \cos(t \pi) I + \sin(t \pi) H \rVert_p ,
}
where $I$ is the identity operator. For $t = \tfrac{1}{2}$ this coincides with Titchmarsh's inequality~\eqref{eq:titchmarsh}:
\formula[]{
 \lVert \rtt \rVert_p & = \lVert T_{1/2} \rVert_p \ge \lVert H \rVert_p = \cot \frac{\pi}{2 p^*} \, .
}
Conjecture~5.7 in~\cite{Lae} asserts that equality holds in~\eqref{eq:laeng}. For all $t \in (0, 1)$, the right-hand side of~\eqref{eq:laeng} is evaluated explicitly (as a maximum of a certain function of one real variable) in Corollary~4.4 in~\cite{HolKalVer}.

For further history and references related to discrete Hilbert transforms and other classical discrete operators in harmonic analysis including Riesz transforms, we refer the reader to~\cite{ArcDomPet,BanKwa,BanKimKwa,CiaGilRonTorVar,DomOsePet,DomPet,GohKru,Gra16,Pie}. Some further historical remarks can be found in Section~\ref{sec:hist} below. For an early account of Riesz's proof on the boundedness of the Hilbert transform on $L^p(\R)$ and of $\dht_0$ on $\ell^p(\Z)$ before publication in \cite{Rie}, we refer the reader to M.L.~Cartwright's article~\cite{Car} \emph{Manuscripts of Hardy, Littlewood, Marcel Riesz and Titchmarsh}, Sections 6--8.

\smallskip

The remaining part of this article consists of six sections. First, in Section~\ref{sec:pre} we recall known properties of various variants of discrete Hilbert transforms. Section~\ref{sec:hist} briefly discusses other papers that involve related arguments. The idea of the proof of Theorem~\ref{thm:main} is given in Section~\ref{sec:idea}. The final three sections contain the actual proof: in Section~\ref{sec:skeletal} we introduce \emph{skeletal decomposition} of integer powers of $\rtt a_n$, in Section~\ref{sec:norms} we use it to find an inequality involving the norms of $\rtt a_n$ on different $\ell^p$ spaces, and we complete the proof of the main result in Section~\ref{sec:cotangents}.

%
%

\section{Preliminaries}
\label{sec:pre}

In this section we introduce basic notation and we gather some known results. For the convenience of the reader, we also provide short proofs where available.

We always assume that $p \in (1, \infty)$. By $\ell^p$ we denote the class of doubly infinite sequences $(a_n) = (a_n : n \in \Z)$ such that
\formula{
 \lVert (a_n) \rVert_p^p & = \sum_{n \in \Z} \lvert a_n \rvert^p
}
is finite. For a continuous linear operator $\ttt$ on $\ell^p$ we denote by $\lVert \ttt \rVert_p$ the operator norm of $\ttt$. We write a somewhat informal symbol $\ttt a_n$ for the $n$-th entry of the sequence obtained by applying $\ttt$ to the sequence $(a_n)$, and $\lVert a_n \rVert_p$ for the norm of the sequence $(a_n)$. We also commonly use the short-hand notation $\ttt[\ph(n)]$ for the $n$-th entry of the application of the operator $\ttt$ to the sequence obtained by evaluating $\ph(n)$ for $n \in \Z$. To improve readability, we often denote multiplication of numbers $a$ and $b$ by $a \cdot b$ rather than $a b$.

We work with a variant of the Riesz--Titchmarsh operator, studied, among others, by S.~Kak~\cite{Kak}, and thus sometimes called the \emph{Kak--Hilbert transform}:
\formula[]{
 \kht a_n & = \frac{2}{\pi} \sum_{m \in 2 \Z + 1} \frac{a_{n - m}}{m} \, ,
}
and a similarly modified operator $\dht_0$, which we denote simply by $\dht$:
\formula[]{
 \dht a_n & = \frac{2}{\pi} \sum_{m \in 2 \Z \setminus \{0\}} \frac{a_{n - m}}{m} \, .
}
We also commonly use the operator
\formula[]{
 \iii a_n & = \frac{4}{\pi^2} \sum_{m \in 2 \Z + 1} \frac{a_{n - m}}{m^2} \, .
}
Note that all these operators are convolution operators with kernel in $\ell^p$ for every $p \in (1, \infty)$. In particular, by Hölder's inequality, $\kht a_n$, $\dht a_n$ and $\iii a_n$ are well-defined when $(a_n)$ is in $\ell^p$ for some $p \in (1, \infty)$.

The following (rather well-known) result allows us to write Theorem~\ref{thm:main} in terms of $\kht$ rather than $\rtt$, and Theorem~\ref{thm:bk} in terms of $\dht$ rather than $\dht_0$. The first equality, $\lVert \kht \rVert_p = \lVert \rtt \rVert_p$, is proved in Theorem~3.1 in~\cite{CarSam}.

\begin{lemma}
\label{lem:norms}
We have the following equalities of operator norms:
\formula{
 \lVert \kht \rVert_p & = \lVert \rtt \rVert_p , \qquad \lVert \dht \rVert_p = \lVert \dht_0 \rVert_p , \qquad \lVert \iii \rVert_p = 1 .
}
\end{lemma}

\begin{proof}
Let $(a_n)$ be a sequence in $\ell^p$, and write $b_n = a_{2 n}$ and $c_n = a_{2 n - 1}$. Then
\formula{
 \rtt b_n & = \frac{1}{\pi} \sum_{m \in \Z} \frac{b_{n - m}}{m + \tfrac12} = \frac{2}{\pi} \sum_{m \in \Z} \frac{a_{2 n - 2 m}}{2 m + 1} = \frac{2}{\pi} \sum_{k \in 2 \Z + 1} \frac{a_{2 n + 1 - k}}{k} = \kht a_{2 n + 1} , \\
 \rtt c_n & = \frac{1}{\pi} \sum_{m \in \Z} \frac{c_{n - m}}{m + \tfrac12} = \frac{2}{\pi} \sum_{m \in \Z} \frac{a_{2 n - 2 m - 1}}{2 m + 1} = \frac{2}{\pi} \sum_{k \in 2 \Z + 1} \frac{a_{2 n - k}}{k} = \kht a_{2 n} .
}
It follows that
\formula{
 \lVert \kht a_n \rVert_p^p & = \lVert \kht a_{2 n + 1} \rVert_p^p + \lVert \kht a_{2 n} \rVert_p^p \\
 & = \lVert \rtt b_n \rVert_p^p + \lVert \rtt c_n \rVert_p^p \\
 & \le \lVert \rtt \rVert_p^p (\lVert b_n \rVert_p^p + \lVert c_n \rVert_p^p) \\
 & = \lVert \rtt \rVert_p^p (\lVert a_{2 n} \rVert_p^p + \lVert a_{2 n - 1} \rVert_p^p) = \lVert \rtt \rVert_p^p \lVert a_n \rVert_p^p ,
}
and hence $\lVert \kht \rVert_p \le \lVert \rtt \rVert_p$. On the other hand, if $a_{2 n - 1} = c_n = 0$ for every $n \in \Z$, then $\kht a_{2n} = 0$ for every $n \in \Z$, and we find that
\formula{
 \lVert \rtt b_n \rVert_p & = \lVert \kht a_{2 n + 1} \rVert_p = \lVert \kht a_n \rVert_p \le \lVert \kht \rVert_p \lVert a_n \rVert_p = \lVert \kht \rVert_p \lVert b_n \rVert_p ,
}
so that $\lVert \rtt \rVert_p \le \lVert \kht \rVert_p$. We have thus proved that $\lVert \rtt \rVert_p = \lVert \kht \rVert_p$. A very similar argument shows that $\lVert \dht_0 \rVert_p = \lVert \dht \rVert_p$.

The operator $\iii$ is the convolution with a probability kernel:
\formula{
 \sum_{m \in 2 \Z + 1} \frac{1}{m^2} & = 2 \sum_{k = 0}^\infty \frac{1}{(2 k + 1)^2} = \frac{2 \pi^2}{8} = \frac{\pi^2}{4} \, ,
}
and so $\lVert \iii \rVert_p \le 1$ follows from Jensen's inequality. Finally, if $a_n = 1$ when $\lvert n \rvert \le N$ and $a_n = 0$ otherwise, then a simple estimate shows that for every positive integer $k$ we have
\formula{
 \limsup_{N \to \infty} \frac{\lVert \iii a_n \rVert_p^p}{\lVert a_n \rVert_p^p} & \ge \limsup_{N \to \infty} \frac{1}{2 N + 1} \sum_{n = -N + k}^{N - k} (\iii a_n)^p \\
 & \ge \limsup_{N \to \infty} \frac{2 N + 1 - 2 k}{2 N + 1} \times \frac{4}{\pi^2} \sum_{m = -k}^k \frac{1}{m^2} = \frac{4}{\pi^2} \sum_{m = -k}^k \frac{1}{m^2} ,
}
and the right-hand side can be arbitrarily close to $1$. Thus, $\lVert \iii \rVert_p \ge 1$, and the proof is complete.
\end{proof}

Lemma~\ref{lem:norms}, Theorems~\ref{thm:main} and~\ref{thm:bk}, and formulas~\eqref{eq:titchmarsh} and~\eqref{eq:pichorides}, give the following result.

\begin{corollary}
\begin{enumerate}[label=(\roman*)]
\item For $p \in (1, \infty)$, we have
\formula{
 \cot \frac{\pi}{2p^*} & = \lVert \dht_0 \rVert_p = \lVert \dht \rVert_p \le \lVert \rtt \rVert_p = \lVert \kht \rVert_p .
}
\item For $p = 2 n$ or $p = \frac{2 n}{2 n - 1}$, $n \in \N$, we have
\formula{
 \cot \frac{\pi}{2p^*} & = \lVert \dht_0 \rVert_p = \lVert \dht \rVert_p = \lVert \rtt \rVert_p = \lVert \kht \rVert_p .
}
\end{enumerate}
\end{corollary}

For a sequence $(a_n)$ in $\ell^2$, we denote by $\fourier a(t)$ the Fourier series with coefficients~$a_n$:
\formula{
 \fourier a(t) & = \sum_{n \in \Z} a_n e^{-i n t}
}
for $t \in \R$. As this is a $2 \pi$-periodic function, we only consider $t \in (-\pi, \pi)$.

The following three results are also rather well-known.

\begin{lemma}
\label{lem:fourier}
For every sequence $(a_n)$ in $\ell^2$ we have:
\formula{
 \fourier[\kht a](t) & = -i \sign t \, \fourier a(t) , \\
 \fourier[\dht a](t) & = (1 - \tfrac2\pi \lvert t \rvert) (-i \sign t) \fourier a(t), \\
 \fourier[\iii a](t) & = (1 - \tfrac2\pi \lvert t \rvert) \fourier a(t) .
}
\end{lemma}

\begin{proof}
Observe that $\kht$, $\dht$ and $\iii$ are convolution operators, with kernels in $\ell^2$ (and even in $\ell^1$ for $\iii$). Thus, they are Fourier multipliers, and the corresponding symbols are simply Fourier series with coefficients given by convolution kernels of $\kht$, $\dht$ and $\iii$. Evaluation of these kernels reduces to well-known formulas.

For $\kht$, we have
\formula{
 \frac{1}{2 \pi} \int_{-\pi}^\pi (-i \sign t) e^{i m t} dt & =
 \begin{cases}
  \frac{2}{\pi} \, \frac{1}{m} & \text{if $m$ is odd,} \\
  0 & \text{if $m$ is even.}
 \end{cases}
}
For $\dht$,
\formula{
 \frac{1}{2 \pi} \int_{-\pi}^\pi (1 - \tfrac2\pi \lvert t \rvert) (-i \sign t) e^{i m t} dt & =
 \begin{cases}
  0 & \text{if $m$ is odd or $m = 0$,} \\
  \frac{2}{\pi} \, \frac{1}{m} & \text{if $m$ is even and $m \ne 0$.}
 \end{cases}
}
Finally, for $\iii$,
\formula{
 \frac{1}{2 \pi} \int_{-\pi}^\pi (1 - \tfrac2\pi \lvert t \rvert) e^{i m t} dt & =
 \begin{cases}
  \frac{4}{\pi^2} \, \frac{1}{m^2} & \text{if $m$ is odd,} \\
  0 & \text{if $m$ is even.}
 \end{cases}
 \qedhere
}
\end{proof}

\begin{lemma}
\label{lem:l2}
We have
\formula{
 \lVert \kht \rVert_2 & = \lVert \dht \rVert_2 = 1 .
}
\end{lemma}

\begin{proof}
This is an immediate consequence of Lemma~\ref{lem:fourier} and Parseval's identity.
\end{proof}

\begin{lemma}
\label{lem:dki}
For every sequence $(a_n)$ in $\ell^p$ we have
\formula{
 \dht a_n & = \iii \kht a_n = \kht \iii a_n .
}
\end{lemma}

We offer two independent proofs of this simple result.

\begin{proof}[Proof \#1:]
We have
\formula{
 \iii \kht a_n = \kht \iii a_n & = \frac{8}{\pi^3} \sum_{m \in 2 \Z + 1} \sum_{k \in 2 \Z + 1} \frac{a_{n - m - k}}{m^2 k} \\
 & = \frac{8}{\pi^3} \sum_{m \in 2 \Z + 1} \sum_{j \in 2 \Z} \frac{a_{n - j}}{m^2 (j - m)} \\
 & = \frac{8}{\pi^3} \sum_{j \in 2 \Z} \biggl( \sum_{m \in 2 \Z + 1} \frac{1}{m^2 (j - m)} \biggr) a_{n - j}.
}
Note that all sums are absolutely convergent. The inner sum can be evaluated explicitly: we have
\formula{
 \sum_{m \in 2 \Z + 1} \frac{1}{m^2 (j - m)} & = \frac{1}{j} \sum_{m \in 2 \Z + 1} \frac{1}{m^2} + \frac{1}{j^2} \sum_{m \in 2 \Z + 1} \biggl(\frac{1}{m} + \frac{1}{j - m} \biggr) \\
 & = \frac{1}{j} \, \frac{\pi^2}{4} + \frac{1}{j^2} \sum_{k = 0}^\infty \biggl(\frac{1}{j + 2 k + 1} + \frac{1}{j - 2 k - 1}\biggr) \\
 & = \frac{\pi^2}{4} \, \frac{1}{j} + \frac{1}{j^2} \, \frac{\pi}{2} \tan \frac{j \pi}{2} \, .
}
Since for $j \in 2 \Z$ we have $\tan \frac{j \pi}{2} = 0$, it follows that
\formula{
 \iii \kht a_n = \kht \iii a_n & = \frac{8}{\pi^3} \sum_{j \in 2 \Z} \frac{\pi^2}{4} \, \frac{1}{j} \, a_{n - j} = \dht a_n ,
}
as desired.
\end{proof}

\begin{proof}[Proof \#2:]
By Lemma~\ref{lem:fourier}, if $(a_n)$ is in $\ell^2$, then
\formula{
 \fourier[\iii \kht a](t) & = (1 - \tfrac2\pi \lvert t \rvert) \fourier[\kht a](t) = -i (1 - \tfrac2\pi \lvert t \rvert) \sign t \, \fourier a(t) = \fourier[\dht a](t)
}
and, similarly,
\formula{
 \fourier[\kht \iii a](t) & = -i \sign t \, \fourier[\iii a](t) = -i (1 - \tfrac2\pi \lvert t \rvert) \sign t \, \fourier a(t) = \fourier[\dht a](t) ,
}
as desired. Extension to $\ell^p$ follows by continuity.
\end{proof}

The next lemma provides a key \emph{product rule}. It is similar to an analogous result for $\dht a_n \cdot \dht b_n$, which was used to estimate the operator norm of $\dht$ on $\ell^p$ when $p$ is a power of~$2$ in~\cite{Lae}. The corresponding result for the continuous Hilbert transform is the well-known formula
\formula[eq:square:con]{
 H f \cdot H g & = H(f \cdot H g) + H(H f \cdot g) + f \cdot g ;
}
see pp.~253--254 in~\cite{Gra14}.

\begin{lemma}
\label{lem:square}
For all sequences $(a_n)$ and $(b_n)$ in $\ell^p$ we have
\formula[eq:square]{
 \kht a_n \cdot \kht b_n & = \kht[a_n \cdot \dht b_n] + \kht[\dht a_n \cdot b_n] + \iii[a_n b_n] .
}
\end{lemma}

Again, we offer two proofs of this result.

\begin{proof}[Proof \#1:]
We have
\formula{
 \kht[a_n \cdot \dht b_n] & = \frac{4}{\pi^2} \sum_{m \in 2 \Z + 1} \sum_{k \in 2 \Z \setminus \{0\}} \frac{a_{n - m} b_{n - m - k}}{m k} \\
 & = \frac{4}{\pi^2} \sum_{m \in 2 \Z + 1} \sum_{j \in (2 \Z + 1) \setminus \{m\}} \frac{a_{n - m} b_{n - j}}{m (j - m)} \, .
}
Similarly,
\formula{
 \kht[\dht a_n \cdot b_n] & = \frac{4}{\pi^2} \sum_{m \in 2 \Z + 1} \sum_{j \in (2 \Z + 1) \setminus \{m\}} \frac{a_{n - j} b_{n - m}}{m (j - m)} \, .
}
Changing the roles of $m$ and $j$ in the latter equality and combining both identities, we find that
\formula{
 & \kht[a_n \cdot \dht b_n] + \kht[\dht a_n \cdot b_n] \\
 & \qquad = \frac{4}{\pi^2} \sum_{m \in 2 \Z + 1} \sum_{j \in (2 \Z + 1) \setminus \{m\}} \biggl( \frac{1}{m (j - m)} + \frac{1}{j (m - j)} \biggr) a_{n - m} b_{n - j} \\
 & \qquad = \frac{4}{\pi^2} \sum_{m \in 2 \Z + 1} \sum_{j \in (2 \Z + 1) \setminus \{m\}} \frac{a_{n - m} b_{n - j}}{m j} \, .
}
For $j = m$ (which is excluded from the double sum), the expression under the sum is equal to $a_{n - m} b_{n - m} / m^2$, appearing in the expression for $\iii[a_n \cdot b_n]$. It follows that
\formula{
 & \kht[a_n \cdot \dht b_n] + \kht[\dht a_n \cdot b_n] + \iii[a_n \cdot b_n] = \frac{4}{\pi^2} \sum_{m \in 2 \Z + 1} \sum_{j \in 2 \Z + 1} \frac{a_{n - m} b_{n - j}}{m j} \\
 & \hspace*{8em} = \biggl(\frac{2}{\pi} \sum_{m \in 2 \Z + 1} \frac{a_{n - m}}{m} \biggr) \biggl( \frac{2}{\pi} \sum_{j \in 2 \Z + 1} \frac{b_{n - j}}{j} \biggr) = \kht a_n \cdot \kht b_n ,
}
as desired.
\end{proof}

\begin{proof}[Proof \#2:]
Suppose that $(a_n)$ and $(b_n)$ are in $\ell^2$. Recall that
\formula{
 a_n & = \frac{1}{2 \pi} \int_{-\pi}^\pi \fourier a(t) e^{i n t} dt , & b_n & = \frac{1}{2 \pi} \int_{-\pi}^\pi \fourier b(s) e^{i n s} ds .
}
Therefore,
\formula{
 a_n \cdot b_n & = \frac{1}{4 \pi^2} \int_{-\pi}^\pi \int_{-\pi}^\pi \fourier a(t) \fourier b(s) e^{i n t + i n s} dt ds \\
 & = \frac{1}{2 \pi} \int_{-\pi}^\pi \biggl( \frac{1}{2 \pi} \int_{-\pi}^\pi \fourier a(t) \fourier b(r - t) dt \biggr) e^{i n r} dr ,
}
and thus
\formula{
 \fourier[a \cdot b](r) & = \frac{1}{2 \pi} \int_{-\pi}^\pi \fourier a(t) \fourier b(r - t) dt .
}
Let $S(t)$ and $I(t)$ denote $2\pi$-periodic functions such that $S(t) = \sign t$ and $I(t) = 1 - \tfrac2\pi \lvert t \rvert$ for $t \in (-\pi, \pi)$. Applying the above formula to $(\kht a_n)$ and $(\kht b_n)$ rather than to $(a_n)$ and $(b_n)$, we find that
\formula{
 \fourier[\kht a \cdot \kht b](r) & = \frac{1}{2 \pi} \int_{-\pi}^\pi (-i S(t)) (-i S(r - t)) \fourier a(t) \fourier b(r - t) dt \\
 & = -\frac{1}{2 \pi} \int_{-\pi}^\pi S(t) S(r - t) \cdot \fourier a(t) \fourier b(r - t) dt .
}
Similarly,
\formula{
 \fourier[\kht [a \cdot \dht b]](r) & = -\frac{1}{2 \pi} \int_{-\pi}^\pi S(r) I(r - t) S(r - t) \cdot \fourier a(t) \fourier b(r - t) dt , \\
 \fourier[\kht [\dht a \cdot b]](r) & = -\frac{1}{2 \pi} \int_{-\pi}^\pi S(r) I(t) S(t) \cdot \fourier a(t) \fourier b(r - t) dt , \\
 \fourier[\iii [a \cdot b]](r) & = \frac{1}{2 \pi} \int_{-\pi}^\pi I(r) \cdot \fourier a(t) \fourier b(r - t) dt .
}
The desired result follows now from the elementary identity
\formula[eq:fourier:square]{
 S(t) S(s) & = S(t + s) I(s) S(s) + S(t + s) I(t) S(t) - I(s + t)
}
after substituting $s = r - t$. The proof of~\eqref{eq:fourier:square} is straightforward, but tedious. Both sides are $2\pi$-periodic with respect to both $t$ and $s$, so we can restrict our attention to $t, s \in (-\pi, \pi)$. Additionally, both changing $(t,s)$ to $(-t,-s)$ or to $(s,t)$ does not change neither side of~\eqref{eq:fourier:square}, so it is sufficient to consider the case when $t < s$ and $t + s > 0$. If $0 < t < s < \pi$, then
\formula{
 & S(t + s) I(s) S(s) + S(t + s) I(t) S(t) - I(s + t) \\
 & \qquad = 1 \cdot (1 - \tfrac2\pi s) \cdot 1 + 1 \cdot (1 - \tfrac2\pi t) \cdot 1 - (1 - \tfrac2\pi (t + s)) \\
 & \qquad \qquad = 1 - \tfrac2\pi s + 1 - \tfrac2\pi t - 1 + \tfrac2\pi (t + s) = 1 = S(t) S(s) .
}
Similarly, if $-\pi < t < 0 < s < \pi$ and $t + s > 0$, then
\formula{
 & S(t + s) I(s) S(s) + S(t + s) I(t) S(t) - I(s + t) \\
 & \qquad = 1 \cdot (1 - \tfrac2\pi s) \cdot 1 + 1 \cdot (1 + \tfrac2\pi t) \cdot (-1) - (1 - \tfrac2\pi (t + s)) \\
 & \qquad \qquad = 1 - \tfrac2\pi s - 1 - \tfrac2\pi t - 1 + \tfrac2\pi (t + s) = -1 = S(t) S(s) .
}
This exhausts all possibilities, and hence the proof is complete. Extension to general $(a_n)$ and $(b_n)$ in $\ell^p$ follows by continuity.
\end{proof}

%
%

\section{Historical remarks}
\label{sec:hist}

The key idea of our proof goes back to the original work of Titchmarsh~\cite{Tit26}. Formula~(2.32) in~\cite{Tit26} asserts that if $b_n = \rtt a_{-n}$, then
\formula{
 \rtt[b_{-n}^2] & = \dht_0[a_n^2] + 2 a_n \cdot \dht_0 a_n .
}
This identity is proved by a direct calculation, and it is essentially equivalent to our Lemma~\ref{lem:square}. It is used to prove that if $\rtt$ is bounded on $\ell^p$, then it is bounded on $\ell^{2 p}$, and in fact
\formula[eq:titchmarsh:square]{
 \lVert \rtt \rVert_{2 p} & \le \lVert \rtt \rVert_p + \sqrt{\tfrac2\pi (5 p + 3) \lVert \rtt \rVert_p + 2 \lVert \rtt \rVert_p^2} \, ;
}
see formula~(2.37) in~\cite{Tit26}. Since $\rtt$ is a unitary operator on $\ell^2$, by induction it follows that $\rtt$ is bounded on $\ell^p$ when $p$ is a power of $2$.

The estimate~\eqref{eq:titchmarsh:square} is not sharp, and thus it led to suboptimal bounds on the norm of $\rtt$. A very similar argument, but with sharp bounds, was applied in~\cite{GohKru} for the continuous Hilbert transform. In this work formula~\eqref{eq:square:con} was used to prove the optimal bound $\lVert H \rVert_p \le \cot \frac{\pi}{2p^*}$ when $p$ or the conjugate exponent of $p$ is a power of $2$. Let us briefly discuss the argument used in~\cite{GohKru}.

By~\eqref{eq:square:con}, we have
\formula{
 \lvert H f \rvert^2 & = \lvert f \rvert^2 + 2 H[f \cdot H f] .
}
Combining this with Hölder's inequality, we find that
\formula{
 \lVert H f \rVert_{2 p}^2 = \lVert (H f)^2 \rVert_p & \le \lVert f^2 \rVert_p + 2 \lVert H[f \cdot H f] \rVert_p \\
 & \le \lVert f \rVert_{2 p}^2 + 2 \lVert H \rVert_p \lVert f \cdot H f \rVert_p \\
 & \le \lVert f \rVert_{2 p}^2 + 2 \lVert H \rVert_p \lVert f \rVert_{2 p} \lVert H f \rVert_{2 p} \\
 & \le \lVert f \rVert_{2 p}^2 + 2 \lVert H \rVert_p \lVert H \rVert_{2 p} \lVert f \rVert_{2 p}^2 ,
}
or, equivalently,
\formula{
 \lVert H \rVert_{2 p} & \le \lVert H \rVert_p + \sqrt{1 + \lVert H \rVert_p^2} .
}
The claimed bound $\lVert H \rVert_p \le \cot \frac{\pi}{2 p^*}$ for $p = 2^n$ follows from $\lVert H \rVert_2 = 1$ by induction on $n$ and the trigonometric identity
\formula{
 \cot \tfrac\alpha2 & = \cot \alpha + \sqrt{1 + \cot^2 \alpha} .
}

As already mentioned, essentially the same argument was used in~\cite{Lae} for the sharp upper bound of the operator norm of $\dht$ on $\ell^p$ when $p$ is a power of $2$. The proof of the similar estimate for $\rtt$ in~\cite{GohKru,GohKre} follows quite a different path: sequences $a_n$ and $\rtt a_n$ are identified with eigenvalues of the real and imaginary parts of an appropriate Volterra operator, and Matsaev's theorem for the latter (Theorem~III.6.2 in~\cite{GohKre}) yields the desired inequality. Intestingly, however, the key estimate in the proof of Matsaev's theorem is obtained using the very same idea, attributed in~\cite{GohKre} to M.~Cotlar (see page~159 in~\cite{Cot}): an analogue of Lemma~\ref{lem:square} for Volterra operators (Theorem~III.3.1 in~\cite{GohKre}) is used to derive a bound for exponent $2 p$ from an estimate for exponent $p$.

In our case, in order to estimate the operator norm of $\kht$ on $\ell^{2 k}$, we develop the expression $(\kht a_n)^k$ repeatedly using Lemma~\ref{lem:square}. This idea can be traced back to the original Riesz's work~\cite{Rie}, where the development of $(f + i H f)^k$ was used to find the suboptimal bound
\formula{
 \lVert H \rVert_{2 k} & \le \frac{2 k}{\log 2}
}
for the continuous Hilbert transform for even integers $p = 2 k$, see Sections~3 and~7 in~\cite{Rie}. As a further historical note we mention that in response to a letter from Riesz dated December~23, 1923, describing his proof, Hardy writes back on January~5, 1924 and says: ``Most elegant \& beautiful, it is amazing that none of us would have seen it before (even for $p=4$!),'' see pp.~489~\&~502 in Cartwright's article~\cite{Car}.

%
%

\section{Idea of the proof}
\label{sec:idea}

As mentioned above, the key idea of our work is to develop the expression $(\kht a_n)^k$ repeatedly using the product rule given in Lemma~\ref{lem:square}. The difficult part is to do it in a right way, and know when to stop. Here is a brief description of our method.

In the first step, we simply take two factors $\kht a_n$ from $(\kht a_n)^k$, we apply the product rule~\eqref{eq:square} to it, and we expand the resulting expression. In each of the next steps, our expression is a finite sum of finite products, and we apply repeatedly the following procedure to each of the increasing number of summands.
\begin{itemize}
\item Whenever there is no factor of the form $\kht[\ldots]$ other than $\kht a_n$, or the entire summand is of the form $\kht[\ldots]$, we stop.
\item Otherwise there is exactly one factor of the form $\kht[\ldots]$ other than $\kht a_n$. We choose this factor and one of the remaining factors $\kht a_n$, we apply the product rule~\eqref{eq:square}, and expand the result.
\end{itemize}
The above procedure leads to a finite sum of terms which are either of the form $\kht[\ldots]$, or of the form $\iii[\ldots] (\kht a_n)^j$ with no other appearance of $\kht$.

What remains to be done is to use Hölder's inequality and known expressions for the operator norms of $\dht$ and $\iii$ to bound the $\ell^{p/k}$ norm of $(\kht a_n)^k$, and then simplify the result. We illustrate the above approach by explicit calculations for $k = 2, 3, 4$.

For $k = 2$ we simply have
\formula{
 (\kht a_n)^2 & = \kht a_n \cdot \kht a_n = 2 \kht[a_n \cdot \dht a_n] + \iii[a_n^2] .
}
Therefore, whenever $\lVert a_n \rVert_p \le 1$, we have
\formula{
 \lVert \kht a_n \rVert_p^2 = \lVert (\kht a_n)^2 \rVert_{p/2} & \le 2 \lVert \kht \rVert_{p/2} \lVert \dht \rVert_p + \lVert \iii \rVert_{p/2} .
}
We already know that $\lVert \iii \rVert_p = 1$ and $\lVert \dht \rVert_p = \cot \tfrac\pi{2p}$ for every $p \ge 2$. Thus, if $\lVert \kht \rVert_{p/2} = \cot \tfrac\pi p$ for some $p \ge 4$, then
\formula[eq:k2]{
 \lVert \kht \rVert_p^2 & \le 2 \cot \tfrac\pi p \cdot \cot \tfrac\pi{2p} + 1 = (\cot \tfrac\pi{2p})^2 .
}
In fact, the above argument easily implies that the operator norm of $\kht$ on $\ell^p$ is $\cot \tfrac\pi{2p}$ when $p$ is a power of $2$. This is essentially the same argument as the one used for the continuous Hilbert transform in~\cite{GohKru} and described in Section~\ref{sec:hist}, as well as the one applied in~\cite{Lae} for the discrete Hilbert transform $\dht_0$.

\begin{remark}
\label{rem:k2}
The above argument shows that a sharp estimate on $\ell^{p/2}$ implies sharp estimate on $\ell^p$, and it does not really require the result of~\cite{BanKwa}. Indeed: even without that result we know that $\lVert \dht \rVert_p \le \lVert \iii \rVert_p \lVert \kht \rVert_p = \lVert \kht \rVert_p$, and so instead of~\eqref{eq:k2} we obtain that
\formula{
 \lVert \kht \rVert_p^2 & \le 2 \cot \tfrac\pi p \cdot \lVert \kht \rVert_p + 1 .
}
This is sufficient to show that if $\lVert \kht \rVert_{p/2} \le \cot \tfrac\pi p$, then $\lVert \kht \rVert_p \le \cot \tfrac\pi{2p}$.
\end{remark}

The calculations get more involved for $k = 3$. In this case we have
\formula{
 (\kht a_n)^3 & = \kht a_n \cdot \bigl(\kht a_n \cdot \kht a_n) \\
 & \stackrel r= \kht a_n \cdot \bigl(2 \kht[a_n \cdot \dht a_n] + \iii[a_n^2]\bigr) \\
 & \stackrel e= 2 \kht a_n \cdot \kht[a_n \cdot \dht a_n] + \kht a_n \cdot \iii[a_n^2] \\
 & \stackrel r= 2 \kht[a_n \cdot \dht[a_n \cdot \dht a_n]] + 2 \kht[\dht a_n \cdot a_n \cdot \dht a_n] + 2 \iii[a_n \cdot a_n \cdot \dht a_n] \\
 & \qquad + \kht a_n \cdot \iii[a_n^2] ,
}
where $\stackrel r=$ marks an application of the product rule~\eqref{eq:square}, and $\stackrel e=$ denotes simple expansion. Therefore, whenever $\lVert a_n \rVert_p \le 1$, we have
\formula{
 \lVert \kht a_n \rVert_p^3 = \lVert (\kht a_n)^3 \rVert_{p/3} & \le 2 \lVert \kht \rVert_{p/3} \lVert \dht \rVert_{p/2} \lVert \dht \rVert_p + 2 \lVert \kht \rVert_{p/3} \lVert \dht \rVert_p \lVert \dht \rVert_p \\
 & \hspace*{8em} + 2 \lVert \iii \rVert_{p/3} \lVert \dht \rVert_p + \lVert \kht \rVert_p \lVert \iii \rVert_{p/2} .
}
If $\lVert \kht \rVert_{p/3} = \cot \tfrac{3\pi}{2p}$, then it follows that
\formula{
 \lVert \kht \rVert_p^3 & \le 2 \cot \tfrac{3\pi}{2p} \cdot \cot \tfrac\pi p \cdot \cot \tfrac\pi{2p} + 2 \cot \tfrac{3\pi}{2p} \cdot \cot \tfrac\pi{2p} \cdot \cot \tfrac\pi{2p} + 2 \cot \tfrac\pi{2p} + \lVert \kht \rVert_p .
}
It is an elementary, but non-obvious fact that this inequality becomes an equality if $\lVert \kht \rVert_p = \cot \tfrac\pi{2p}$, and it is relatively easy to see that the inequality in fact implies that $\lVert \kht \rVert_p \le \cot \tfrac\pi{2p}$.

\begin{remark}\label{rem:k3}
We stress that in the above proof that $\lVert \kht \rVert_{p/3} \le \cot \tfrac{3\pi}{2p}$ implies $\lVert \kht \rVert_p \le \cot \tfrac\pi{2p}$, we use the result of~\cite{BanKwa} in its full strength: we need an estimate for the operator norm of $\dht$ on $\ell^{p/2}$, and the argument would not work without this ingredient. In other words, the approach used for $k = 2$ in Remark~\ref{rem:k2} no longer applies for $k = 3$.
\end{remark}

The case $k = 4$ is, of course, not interesting, as it is sufficient to apply the result for $k = 2$ twice. Nevertheless, it is instructive to study $k = 4$ only to better understand the method. The initial steps are the same as for $k = 3$, and we re-use the above calculation:
\formula{
 (\kht a_n)^4 & = \kht a_n \cdot \bigl(2 \kht[a_n \cdot \dht[a_n \cdot \dht a_n]] + 2 \kht[\dht a_n \cdot a_n \cdot \dht a_n] \\
 & \hspace*{16em} + \iii[a_n^2 \cdot \dht a_n] + \kht a_n \cdot \iii[a_n^2]\bigr) \\
 & \stackrel e= 2 \kht a_n \cdot \kht[a_n \cdot \dht[a_n \cdot \dht a_n]] \\
 & \qquad + 2 \kht a_n \cdot \kht[\dht a_n \cdot a_n \cdot \dht a_n] \\
 & \qquad \qquad + \kht a_n \cdot \iii[a_n^2 \cdot \dht a_n] + (\kht a_n)^2 \cdot \iii[a_n^2] \displaybreak[0]\\
 & \stackrel r= 2 \kht[a_n \cdot \dht[a_n \cdot \dht[a_n \cdot \dht a_n]]] + 2 \kht[\dht a_n \cdot a_n \cdot \dht[a_n \cdot \dht a_n]] \\
 & \hspace*{16.5em} + 2 \iii[a_n \cdot a_n \cdot \dht[a_n \cdot \dht a_n]] \\
 & \qquad + 2 \kht[a_n \cdot \dht[\dht a_n \cdot a_n \cdot \dht a_n]] + 2 \kht[\dht[a_n] \cdot \dht a_n \cdot a_n \cdot \dht a_n] \\
 & \qquad \hspace*{15em} + 2 \iii[a_n \cdot \dht a_n \cdot a_n \cdot \dht a_n] \\
 & \qquad \qquad + \kht a_n \cdot \iii[a_n^2 \cdot \dht a_n] + (\kht a_n)^2 \cdot \iii[a_n^2] .
}
Again this leads to an inequality that involves only $\lVert \kht \rVert_p$ and $\lVert \kht \rVert_{p/4}$ and known constants, and solving this inequality shows that if $\lVert \kht \rVert_{p/4} \le \cot \tfrac{2\pi}p$, then $\lVert \kht \rVert_p \le \cot \tfrac\pi{2p}$. However, the calculations are even more involved than in the case $k = 3$, and so we stop here.

\begin{remark}
It is important not to develop further the expressions such as $(\kht a_n)^{k - 2} \cdot \iii[a_n^2]$ using the product rule~\eqref{eq:square}. For $k = 4$ that would lead to expressions involving the unknown operator norms of $\kht$ on $\ell^{p/2}$. This could be circumvented by re-using appropriately the result for $k = 2$. However, already when $k = 5$ further development of $(\kht a_n)^3 \cdot \iii[a_n^2]$ would lead to expressions involving the unknown operator norms of $\kht$ on $\ell^{p/3}$ or $\ell^{p/2}$, and the method would break.
\end{remark}

Clearly, the expressions become rapidly more complicated as $k$ grows, and we need a systematic way to handle them. This is done using \emph{skeletons} introduced in the next section.

%
%

\section{Skeletons, frames and buildings}
\label{sec:skeletal}

The following auxiliary definition allows us to conveniently enumerate the terms in the development of $(\kht a_n)^k$.

\begin{definition}
\label{def:skeleton}
For a positive integer $k$, we define the set $\skel_k$ of \emph{skeletons} of size $k$ inductively:
\formula{
 \skel_1 & = \bigl\{ \{1\} \bigr\} , \\
 \skel_{k + 1} & = \bigl\{ \{S, k + 1\} : S \in \skel_k \bigr\} \cup \bigl\{ S \cup \{\{k + 1\}\} : S \in \skel_k \bigr\} .
}
We define the set $\fram$ of \emph{frames} to be the minimal collection of numbers and sets with the following properties:
\begin{itemize}
\item every integer is a frame;
\item every finite set of frames is a frame.
\end{itemize}
More precisely, we let $\fram_0$ be the set of integers, that is, \emph{frames of depth $0$}, and for every positive integer $k$ we define the set $\fram_k$ of \emph{frames of depth at most $k$} to be the collection of all finite sets of frames of depth less than $k$. Finally, we define the \emph{size} $\lvert F \rvert$ of a frame $F$ inductively by
\formula{
 \lvert F \rvert & =
 \begin{cases}
  1 & \text{if $F$ is an integer,} \\[0.2em]
  \displaystyle \sum_{f \in F} \lvert f \rvert & \text{if $F$ is a set.}
 \end{cases}
}
\end{definition}

Clearly, there are $2^k$ skeletons of size $k$, and every skeleton is a frame. The size of a frame $F$ is just the number of integers (`bones') that appear in a textual (roster notation) or graphical (Venn diagram) representation of $F$, while the depth of $F$ is the maximal number of nested brackets in the textual representation of $F$. It is straightforward to see that every skeleton of size $k$ is a frame of size $k$ (with variable depth). By inspection, we find that the initial sets of skeletons are:
\formula{
 \skel_1 & = \bigl\{ \{1\} \bigr\} \\
 \skel_2 & = \bigl\{ \{\{1\}, 2\} , \{1, \{2\}\} \bigr\} \\
 \skel_3 & = \bigl\{ \{\{\{1\}, 2\}, 3\} , \{\{1, \{2\}\}, 3\}, \{\{1\}, 2, \{3\}\} , \{1, \{2\}, \{3\}\} \bigr\} \\
 \skel_4 & = \bigl\{ \{\{\{\{1\}, 2\}, 3\}, 4\} , \{\{\{1, \{2\}\}, 3\}, 4\}, \{\{\{1\}, 2, \{3\}\}, 4\} , \\
 & \qquad \{\{1, \{2\}, \{3\}\}, 4\} , \{\{\{1\}, 2\}, 3, \{4\}\} , \{\{1, \{2\}\}, 3, \{4\}\} , \\
 & \qquad \qquad \{\{1\}, 2, \{3\}, \{4\}\} , \{1, \{2\}, \{3\}, \{4\}\} \bigr\} .
}

\begin{definition}
\label{def:building}
For a sequence $(a_n)$ in $\ell^p$ and a frame $F$, we define the \emph{building} $(\dht^{\{F\}} a_n)$ with frame $\{F\}$ inductively by
\formula{
 \dht^{\{F\}} a_n & =
 \begin{cases}
  a_n & \text{if $F$ is an integer,} \\
  \displaystyle \dht \biggl[ \prod_{f \in F} \dht^{\{f\}} a_n \biggr] & \text{if $F$ is a set.}
 \end{cases}
}
If the frame $F$ is a set containing more than one element, we extend the above definition, so that the building $(\dht^F a_n)$ with frame $F$ is given by
\formula{
 \dht^F a_n & = \prod_{f \in F} \dht^{\{f\}} a_n .
}
\end{definition}

In other words, construction of the building $(\dht^F a_n)$ with frame $F$ corresponds to replacing every integer in the textual representation of the frame $F$ by $(a_n)$, every pair of corresponding curly brackets $\{\ldots\}$ --- except the outermost ones --- by $\dht[\ldots]$, and every comma by multiplication. For example,
\formula{
 \dht^{\{1\}} a_n & = a_n , \\
 \dht^{\{\{1\}, 2\}} a_n & = \dht a_n \cdot a_n , \\
 \dht^{\{\{\{1\}, 2\}, 3\}} a_n & = \dht[\dht a_n \cdot a_n] \cdot a_n , \\
 \dht^{\{\{1\}, 2, \{3\}\}} a_n & = \dht a_n \cdot a_n \cdot \dht a_n = a_n \cdot (\dht a_n)^2 , \\
 \dht^{\{\{1, \{2\}\}, 3, \{4\}\}} a_n & = \dht[a_n \cdot \dht a_n] \cdot a_n \cdot \dht a_n .
}

Our next results provides a skeletal decomposition of $(\kht a_n)^k$.

\begin{proposition}
\label{prop:power}
For every sequence $(a_n)$ in $\ell^p$ and every positive integer $k$ we have
\formula[eq:power]{
 (\kht a_n)^k & = \sum_{S \in \skel_k} \kht[\dht^S a_n] + \sum_{j = 1}^{k - 1} \sum_{S \in \skel_j} \iii[a_n \cdot \dht^S a_n] \cdot (\kht a_n)^{k - j - 1} .
}
\end{proposition}

\begin{proof}
We proceed by induction with respect to $k$. For $k = 1$ formula~\eqref{eq:power} takes form
\formula{
 \kht a_n & = \sum_{S \in \skel_1} \kht[\dht^S a_n] ,
}
which is obviously true, because $\skel_1$ has only one skeleton $S = \{1\}$ and $\kht[\dht^{\{1\}} a_n] = \kht a_n$. Suppose that~\eqref{eq:power} holds for some positive integer $k$. Then
\formula*[eq:power:aux]{
 (\kht a_n)^{k + 1} & = (\kht a_n)^k \cdot \kht a_n \\
 & = \sum_{S \in \skel_k} \kht[\dht^S a_n] \cdot \kht a_n + \sum_{j = 1}^{k - 1} \sum_{S \in \skel_j} \iii[a_n \cdot \dht^S a_n] \cdot (\kht a_n)^{k - j} .
}
By Lemma~\ref{lem:square} (the product rule), for every $S \in \skel_k$ we have
\formula{
 \kht[\dht^S a_n] \cdot \kht a_n & = \kht[\dht^S a_n \cdot \dht a_n] + \kht[\dht[\dht^S a_n] \cdot a_n] + \iii[\dht^S a_n \cdot a_n] \\
 & = \kht[\dht^{S \cup \{\{k + 1\}\}} a_n] + \kht[\dht^{\{S, k + 1\}} a_n] + \iii[a_n \cdot \dht^S a_n] .
}
The term $\iii[a_n \cdot \dht^S a_n]$, summed over all $S \in \skel_k$, corresponds to the term $j = k$ in the latter sum in~\eqref{eq:power:aux}. On the other hand, by Definition~\ref{def:skeleton}, the sum of the remaining terms $\kht[\dht^{S \cup \{\{k + 1\}\}} a_n]$ and $\kht[\dht^{\{S, k + 1\}} a_n]$ over $S \in \skel_k$ is equal to the sum of $\kht[\dht^S a_n]$ over $S \in \skel_{k + 1}$. Therefore,
\formula{
 (\kht a_n)^{k + 1} & = \sum_{S \in \skel_{k + 1}} \kht[\dht^S a_n] + \sum_{j = 1}^k \sum_{S \in \skel_j} \iii[a_n \cdot \dht^S a_n] \cdot (\kht a_n)^{k - j} ,
}
as desired. This completes the proof by induction.
\end{proof}

%
%

\section{Norms and constants}
\label{sec:norms}

By Lemma~\ref{lem:norms}, the operator norm of $\iii$ on $\ell^p$ is equal to $1$. By the same lemma and the result of~\cite{BanKwa} (given in Theorem~\ref{thm:bk}), the operator norm of $\dht$ on $\ell^p$ is equal to
\formula{
 C_p & = \cot \tfrac{\pi}{2p^*} =
 \begin{cases}
  \tan \tfrac\pi{2p} & \text{when $1 < p \le 2$,} \\
  \cot \tfrac\pi{2p} & \text{when $2 \le p < \infty$.}
 \end{cases}
}
The following definition will allow us to bound the $\ell^{p/\lvert F \rvert}$ norm of the building $(\dht^F a_n)$ with frame $F$.

\begin{definition}
\label{def:buildingnorm}
For a frame $F$, we define the corresponding \emph{building norm} $C_p^{\{F\}}$ inductively by
\formula{
 C_p^{\{F\}} & =
 \begin{cases}
  1 & \text{if $F$ is an integer,} \\
  \displaystyle C_{p / \lvert F \rvert} \prod_{f \in F} C_p^{\{f\}} & \text{if $F$ is a set.}
 \end{cases}
}
If the frame $F$ is a set containing more than one element, we extend the above definition, so that
\formula{
 C_p^F a_n & = \prod_{f \in F} C_p^{\{f\}} .
}
\end{definition}

For example, we have
\formula{
 C_p^{\{1\}} & = 1 , \\
 C_p^{\{\{1\}, 2\}} & = C_p, \\
 C_p^{\{\{\{1\}, 2\}, 3\}} & = C_{p/2} \cdot C_p , \\
 C_p^{\{\{1\}, 2, \{3\}\}} & = C_p \cdot C_p , \\
 C_p^{\{\{\{\{1\}, 2\}, 3\}, 4\}} & = C_{p/3} \cdot C_{p/2} \cdot C_p , \\
 C_p^{\{\{1, \{2\}\}, 3, \{4\}\}} & = C_{p/2} \cdot C_p \cdot C_p .
}

The relation between the norm of the building $\dht^F a_n$ and the corresponding building norm is provided by the following result.

\begin{lemma}
\label{lem:buildingnorm}
For every sequence $(a_n)$ in $\ell^p$ and every frame $F$, we have
\formula{
 \lVert \dht^{\{F\}} a_n \rVert_{p / \lvert F \rvert} & \le C_p^{\{F\}} \lVert a_n \rVert_p^{\lvert F \rvert} .
}
If the frame $F$ is a set, then, more generally,
\formula{
 \lVert \dht^F a_n \rVert_{p / \lvert F \rvert} & \le C_p^F \lVert a_n \rVert_p^{\lvert F \rvert} .
}
\end{lemma}

\begin{proof}
The proof of the first part of the lemma proceeds by induction with respect to the depth of $F$. If $F$ is an integer (a frame of depth $0$), then $\dht^{\{F\}} a_n = a_n$ has norm equal to $\lVert a_n \rVert_p = C_p^{\{F\}} \lVert a_n \rVert_p$. Suppose now that
\formula{
 \lVert \dht^{\{F\}} a_n \rVert_{p / \lvert F \rvert} & \le C_p^{\{F\}} \lVert a_n \rVert_p^{\lvert F \rvert}
}
for every frame $F$ of depth less than $k$ for some positive integer $k$. Let $F$ be a frame of depth $k$. Then
\formula{
 \lVert \dht^{\{F\}} a_n \rVert_{p / \lvert F \rvert} & = \biggl\lVert \dht \biggl[ \prod_{f \in F} \dht^{\{f\}} a_n \biggr] \biggr\rVert_{p / \lvert F \rvert} \le \lVert \dht \rVert_{p / \lvert F \rvert} \, \biggl\lVert \prod_{f \in F} \dht^{\{f\}} a_n \biggr\rVert_{p / \lvert F \rvert} .
}
By definition, the size $\lvert F \rvert$ of frame $F$ is the sum of sizes $\lvert f \rvert$ of all frames $f \in F$. Thus, Hölder's inequality implies that
\formula{
 \lVert \dht^{\{F\}} a_n \rVert_{p / \lvert F \rvert} & \le \lVert \dht \rVert_{p / \lvert F \rvert} \prod_{f \in F} \lVert \dht^{\{f\}} a_n \rVert_{p / \lvert f \rvert} .
}
Using the equality $\lVert \dht \rVert_{p / \lvert F \rvert} = C_{p / \lvert F \rvert}$ and the induction hypothesis, we find that
\formula{
 \lVert \dht^{\{F\}} a_n \rVert_{p / \lvert F \rvert} & \le C_{p / \lvert F \rvert} \prod_{f \in F} C_p^{\{f\}} \lVert a_n \rVert_p^{\lvert f \rvert} = C_p^{\{F\}} \lVert a_n \rVert_p^{\lvert F \rvert} ,
}
as desired. This completes the proof by induction.

The other part of the lemma follows now by the same argument as above, but without the final application of $\dht$:
\formula{
 \lVert \dht^F a_n \rVert_{p / \lvert F \rvert} & = \biggl\lVert \prod_{f \in F} \dht^{\{f\}} a_n \biggr\rVert_{p / \lvert F \rvert} \\
 & \le \prod_{f \in F} \lVert \dht^{\{f\}} a_n \rVert_{p / \lvert f \rvert} = \prod_{f \in F} C_p^{\{f\}} \lVert a_n \rVert_p^{\lvert f \rvert} = C_p^F \lVert a_n \rVert_p^{\lvert F \rvert} ,
}
and the proof is complete.
\end{proof}

The skeletal decomposition allows us to find an inequality that links the operator norms of $\kht$ on $\ell^p$ and $\ell^{p/k}$.

\begin{lemma}
\label{lem:rec}
Let $k$ be a positive integer and $p \ge k$. The operator norms of $\kht$ on $\ell^p$ and $\ell^{p/k}$ satisfy the inequality
\formula{
 \lVert \kht \rVert_p^k & \le \biggl( \sum_{S \in \skel_k} C_p^S \biggr) \lVert \kht \rVert_{p/k} + \sum_{j = 1}^{k - 1} \biggl( \sum_{S \in \skel_j} C_p^S \biggr) \lVert \kht \rVert_p^{k - j - 1} .
}
\end{lemma}

\begin{proof}
Suppose that $(a_n)$ is a sequence in $\ell^p$, $k$ is a positive integer and $p \ge k$. By Proposition~\ref{prop:power} and triangle inequality, we have
\formula{
 \lVert (\kht a_n)^k \rVert_{p/k} & \le \sum_{S \in \skel_k} \lVert \kht[\dht^S a_n] \rVert_{p/k} + \sum_{j = 1}^{k - 1} \sum_{S \in \skel_j} \lVert \iii[a_n \cdot \dht^S a_n] \cdot (\kht a_n)^{k - j - 1} \rVert_{p/k} .
}
The left-hand side is equal to $\lVert \kht a_n \rVert_p^k$. Applying Hölder's inequality to the expressions on the right-hand side, we find that
\formula{
 \lVert \kht a_n \rVert_p^k & \le \sum_{S \in \skel_k} \lVert \kht[\dht^S a_n] \rVert_{p/k} \\
 & \qquad + \sum_{j = 1}^{k - 1} \sum_{S \in \skel_j} \lVert \iii[a_n \cdot \dht^S a_n] \rVert_{p/(j+1)} \lVert (\kht a_n)^{k - j - 1} \rVert_{p/(k - j - 1)} \\
 & = \sum_{S \in \skel_k} \lVert \kht[\dht^S a_n] \rVert_{p/k} + \sum_{j = 1}^{k - 1} \sum_{S \in \skel_j} \lVert \iii[a_n \cdot \dht^S a_n] \rVert_{p/(j+1)} \lVert \kht a_n \rVert_p^{k - j - 1} .
}
Since the operator norm of $\iii$ is equal to $1$, we obtain
\formula{
 \lVert \kht a_n \rVert_p^k & \le \sum_{S \in \skel_k} \lVert \kht \rVert_{p/k} \lVert \dht^S a_n \rVert_{p/k} + \sum_{j = 1}^{k - 1} \sum_{S \in \skel_j} \lVert a_n \cdot \dht^S a_n \rVert_{p/(j+1)} \lVert \kht a_n \rVert_p^{k - j - 1} .
}
Another application of Hölder's inequality leads to
\formula{
 \lVert \kht a_n \rVert_p^k & \le \sum_{S \in \skel_k} \lVert \kht \rVert_{p/k} \lVert \dht^S a_n \rVert_{p/k} + \sum_{j = 1}^{k - 1} \sum_{S \in \skel_j} \lVert a_n \rVert_p \lVert \dht^S a_n \rVert_{p/j} \lVert \kht a_n \rVert_p^{k - j - 1} .
}
By Lemma~\ref{lem:buildingnorm}, we have
\formula{
 \lVert \kht a_n \rVert_p^k & \le \sum_{S \in \skel_k} \lVert \kht \rVert_{p/k} \, C_p^S \lVert a_n \rVert_p^k + \sum_{j = 1}^{k - 1} \sum_{S \in \skel_j} \lVert a_n \rVert_p \, C_p^S \lVert a_n \rVert_p^j \lVert \kht a_n \rVert_p^{k - j - 1} .
}
It remains to apply the supremum over all sequences $(a_n)$ in $\ell^p$ with norm less than or equal to $1$ to both sides of the above inequality.
\end{proof}

%
%

\section{Cotangents}
\label{sec:cotangents}

In order to prove Theorem~\ref{thm:main}, we need one more technical ingredient. We remark that throughout this section the constant $C_p$ is used only with $p \ge 2$, so that $C_p = \cot \tfrac\pi{2p}$.

\begin{lemma}
\label{lem:cot}
If $k$ is a positive integer, $p \ge 2 k$ and $x$ is a real number such that
\formula[eq:cot]{
 x^k & \le \biggl( \sum_{S \in \skel_k} C_p^S \biggr) C_{p/k} + \sum_{j = 1}^{k - 1} \biggl( \sum_{S \in \skel_j} C_p^S \biggr) x^{k - j - 1} ,
}
then $x \le C_p$.
\end{lemma}

\begin{proof}
Let $f_k(x)$ denote the right-hand side of~\eqref{eq:cot} (for a fixed $p \ge 2 k$). Thus, we need to prove that if $x^k \le f_k(x)$, then $x \le C_p$; or, equivalently: if $x > C_p$, then $f_k(x) / x^k < 1$.

The function $f_k(x)$ is a non-zero polynomial of $x$ of degree less than $k$ with non-negative coefficients, and hence $f_k(x) / x^k$ is a decreasing function of $x > 0$. It follows that
\formula{
 \frac{f_k(x)}{x^k} & < \frac{f_k(C_p)}{C_p^k}
}
when $x > C_p$. Thus, it is sufficient to prove that $f_k(C_p) = C_p^k$, that is, for $x = C_p$ we have equality in~\eqref{eq:cot}.

The proof is based on the formula for the cotangent of sum, which we re-write in a form that resembles the product rule~\eqref{eq:square}:
\formula{
 \cot \alpha \cdot \cot \beta & = \cot(\alpha + \beta) \cot \beta + \cot(\alpha + \beta) \cot \alpha + 1
}
whenever $\alpha, \beta > 0$ and $\alpha + \beta < \pi$. Setting $\alpha = \tfrac{i\pi}{2p}$ and $\beta = \tfrac{j\pi}{2p}$ with positive integers $i, j$ such that $p \ge 2 (i + j)$, we find that
\formula[eq:cot:cp]{
 C_{p/i} C_{p/j} & = C_{p/(i+j)} C_{p/j} + C_{p/(i+j)} C_{p/i} + 1 ,
}
which is exactly what is needed to show that $f_k(C_p) = C_p^k$. We proceed by induction with respect to $k$.

For $k = 1$ and $p \ge 2$, we have
\formula{
 f_1(x) & = \biggl( \sum_{S \in \skel_1} C_p^S \biggr) C_p = C_p^{\{1\}} C_p = C_p ,
}
and hence indeed $f_1(C_p) = C_p$. Suppose now that for some positive integer $k$ we have $p \ge 2 k + 2$ and $f_k(C_p) = C_p^k$. Observe that
\formula{
 C_p^{k + 1} = C_p f_k(C_p) & = \biggl( \sum_{S \in \skel_k} C_p^S \biggr) C_{p/k} C_p + \sum_{j = 1}^{k - 1} \biggl( \sum_{S \in \skel_j} C_p^S \biggr) C_p^{k - j} .
}
Recall that every skeleton of size $k + 1$ is equal to either $\{S, k + 1\}$ or $S \cup \{\{k + 1\}\}$ for some skeleton $S$ of size $k$. This property and the definition of the constant $C_p^S$ imply that
\formula{
 f_{k+1}(x) & = \biggl( \sum_{S \in \skel_{k+1}} C_p^S \biggr) C_{p/(k+1)} + \sum_{j = 1}^k \biggl( \sum_{S \in \skel_j} C_p^S \biggr) x^{k - j} \\
 & = \biggl( \sum_{S \in \skel_k} (C_{p/k} C_p^S + C_p^S C_p) \biggr) C_{p/(k+1)} + \sum_{j = 1}^k \biggl( \sum_{S \in \skel_j} C_p^S \biggr) x^{k - j} \\
 & = \biggl( \sum_{S \in \skel_k} C_p^S \biggr) (C_{p/(k+1)} C_{p/k} + C_{p/(k+1)} C_p + 1) + \sum_{j = 1}^{k - 1} \biggl( \sum_{S \in \skel_j} C_p^S \biggr) x^{k - j} .
}
Applying the cotangent of sum identity~\eqref{eq:cot:cp} with $i = 1$ and $j = k$, we obtain
\formula{
 f_{k+1}(x) & = \biggl( \sum_{S \in \skel_k} C_p^S \biggr) C_p C_{p/k} + \sum_{j = 1}^{k - 1} \biggl( \sum_{S \in \skel_j} C_p^S \biggr) x^{k - j} .
}
We conclude that
\formula{
 f_{k+1}(C_p) & = \biggl( \sum_{S \in \skel_k} C_p^S \biggr) C_p C_{p/k} + \sum_{j = 1}^{k - 1} \biggl( \sum_{S \in \skel_j} C_p^S \biggr) C_p^{k - j} = C_p f_k(C_p) = C_p^{k + 1} ,
}
as desired. This completes the proof by induction.
\end{proof}

\begin{corollary}
\label{cor:recursion}
If $k$ is a positive integer, $p \ge 2 k$ and $\lVert \kht \rVert_{p/k} \le C_{p/k}$, then $\lVert \kht \rVert_p \le C_p$.
\end{corollary}

\begin{proof}
By Lemma~\ref{lem:rec}, the number $x = \lVert \kht \rVert_p$ satisfies~\eqref{eq:cot}, and therefore, by Lemma~\ref{lem:cot}, $x \le C_p$.
\end{proof}

\begin{corollary}
\label{cor:main}
If $p$ is a positive even integer, then $\lVert \kht \rVert_p = C_p$.
\end{corollary}

\begin{proof}
By Lemma~\ref{lem:l2}, we have $\lVert \kht \rVert_2 = 1 = C_2$. We apply Corollary~\ref{cor:recursion} with $k = \tfrac p2$ to find that $\lVert \kht \rVert_p \le C_p$. Furthermore, $\lVert \kht \rVert_p = \lVert \rtt \rVert_p \ge C_p$ by~\eqref{eq:titchmarsh}, \eqref{eq:pichorides} and Lemma~\ref{lem:norms}, and so equality follows.
\end{proof}

\begin{proof}[Proof of Theorem~\ref{thm:main}]
Since $\lVert \rtt \rVert_p = \lVert \kht \rVert_p$ by Lemma~\ref{lem:norms}, the desired result when $p$ is an even integer is an immediate consequence of Corollary~\ref{cor:main}. Extension to the case when the conjugate exponent of $p$ is an even integer follows by duality.
\end{proof}

\subsection*{Acknowledgement}
We would like to thank the anonymous referees for their valuable comments and suggestions which greatly improved the exposition of this paper, and for informing us about the connection with Matsaev's theorem and the monograph of Gohberg and Krein.

%
%


\begin{thebibliography}{10}

\bibitem{ArcDomPet}
N.~Arcozzi, K.~Domelevo, and S.~Petermichl,
\emph{Discrete Hilbert transform \`a la Gundy--Varopoulos},
Proc. Amer. Math. Soc. \textbf{148} (2020), no.~6, 2433--2446.
\MR{4080886}

\bibitem{BanKwa}
R.~Ba\~{n}uelos and M.~Kwa\'{s}nicki,
\emph{On the $\ell^p$-norm of the discrete Hilbert transform},
Duke Math. J. \textbf{168} (2019), no.~3, 471--504.
\MR{3909902}

\bibitem{BanKimKwa}
R.~Ba{\~n}uelos, D.~Kim, and M.~Kwa{\'s}nicki,
\emph{Sharp $\ell^p$ inequalities for discrete singular integrals},
arXiv:2209.09737 (2022).

\bibitem{Car}
M.~L.~Cartwright,
\emph{Manuscripts of Hardy, Littlewood, Marcel Riesz and Titchmarsh},
Bull. London Math. Soc. \textbf{14} (1982), no.~6, 472--532.
\MR{679927}

\bibitem{CiaGilRonTorVar}
\'{O}.~Ciaurri, T.~A.~Gillespie, L.~Roncal, J.~L.~Torrea, and J.~L.~Varona,
\emph{Harmonic analysis associated with a discrete Laplacian},
J. Anal. Math. \textbf{132} (2017), 109--131.
\MR{3666807}

\bibitem{Cot}
M.~Cotlar,
\emph{A unified theory of Hilbert transforms and ergodic theorems}.
Rev. Mat. Cuyana \textbf{1} (1956), 105--167.
\MR{84632}

\bibitem{CarSam}
L.~De~Carli and G.~S.~Samad,
\emph{One-parameter groups of operators and discrete Hilbert transforms},
Canad. Math. Bull. \textbf{59} (2016), no.~3, 497--507.
\MR{3563731}

\bibitem{DomOsePet}
K.~Domelevo, A.~Os\k{e}kowski and S.~Petermichl,
\emph{Various sharp estimates for semi-discrete Riesz transforms of the second order}.
In: A.~Baranov, S.~Kisliakov and N.~Nikolski [eds.],
\emph{50 years with Hardy spaces},
Oper. Theory Adv. Appl., vol. 261,
Birkh\"{a}user/Springer, Cham, 2018, pp.~229--255.
\MR{3792098}

\bibitem{DomPet}
K.~Domelevo and S.~Petermichl,
\emph{Sharp $L^p$ estimates for discrete second order Riesz transforms},
Adv. Math. \textbf{262} (2014), 932--952.
\MR{3228446}

\bibitem{Gam}
T.~W.~Gamelin,
\emph{Uniform algebras and Jensen measures},
London Math. Soc. Lecture Note Series 32,
Cambridge University Press, 1978.
\MR{521440}

\bibitem{GohKre}
I.~C.~Gohberg and M.~G.~Kre\u{\i}n,
\emph{Theory and applications of Volterra operators in Hilbert space},
Transl. Math. Monogr., vol. 24,
American Mathematical Society, Providence, RI, 1970.
\MR{264447}

\bibitem{GohKru}
I.~Ts.~Gohberg and N.~Ya.~Krupnik,
\emph{Norm of the Hilbert transformation in the ${L}^p$ space},
Funct. Anal. Pril. \textbf{2} (1968), no.~2, 91--92 [in Russian].
English transl. in Funct. Anal. Appl. \textbf{2} (1968), no.~4, 180--181.

\bibitem{Gra14}
L.~Grafakos,
\emph{Classical Fourier analysis}, third ed.,
Graduate Texts in Mathematics, vol. 249,
Springer, New York, 2014.
\MR{3243734}

\bibitem{Gra16}
\bysame,
\emph{An elementary proof of the square summability of the discrete Hilbert transform},
Amer. Math. Monthly \textbf{101} (2016), no.~5, 456--458.

\bibitem{HolKalVer}
B.~Hollenbeck, N.~J.~Kalton, and I.~E.~Verbitsky,
\emph{Best constants for some operators associated with the Fourier and Hilbert transforms},
Studia Math. \textbf{157} (2003), no.~3, 237--278.
\MR{1980300}

\bibitem{Kak}
S.~Kak,
\emph{The discrete finite Hilbert transform},
Indian J. Pure Appl. Math. \textbf{8} (1977), no.~11, 1385--1390.
\MR{547566}

\bibitem{Lae}
E.~Laeng,
\emph{Remarks on the Hilbert transform and on some families of multiplier operators related to it},
Collect. Math. \textbf{58} (2007), no.~1, 25--44.
\MR{2310545}

\bibitem{Pic}
S.~K.~Pichorides,
\emph{On the best values of the constants in the theorems of M.~Riesz, Zygmund and Kolmogorov},
Studia Math. \textbf{44} (1972), 165--179. (errata insert).
\MR{312140}

\bibitem{Pie}
L.~B.~Pierce,
\emph{Discrete analogues in harmonic analysis},
ProQuest LLC, Ann Arbor, MI, 2009,
Ph.D. Thesis, Princeton University.
\MR{2713096}

\bibitem{Rie}
M.~Riesz,
\emph{Sur les fonctions conjugu\'{e}es},
Math. Z. \textbf{27} (1928), no.~1, 218--244.
\MR{1544909}

\bibitem{Tit26}
E.~C. Titchmarsh,
\emph{Reciprocal formulae involving series and integrals},
Math. Z. \textbf{25} (1926), no.~1, 321--347.
\MR{1544814}

\bibitem{Tit27}
\bysame,
\emph{Reciprocal formulae involving series and integrals (correction)},
Math. Z. \textbf{26} (1927), no.~1, 496.
\MR{1544871}

\end{thebibliography}

\providecommand{\bysame}{\leavevmode\hbox to3em{\hrulefill}\thinspace}
\providecommand{\MR}{\relax\ifhmode\unskip\space\fi MR }
\providecommand{\MRhref}[2]{%
  \href{http://www.ams.org/mathscinet-getitem?mr=#1}{#2}
}
\providecommand{\href}[2]{#2}

%
%

\end{document}